\documentclass[a4paper]{article}

\usepackage[english]{babel}
\usepackage[utf8x]{inputenc}
\usepackage[T1]{fontenc}

\usepackage[a4paper,top=3cm,bottom=3cm,left=4cm,right=4cm,marginparwidth=1.75cm]{geometry}
\usepackage{amsmath,array,amssymb, amsfonts, latexsym,amsthm, mathrsfs, dsfont}
\usepackage{graphicx}
\usepackage[colorinlistoftodos]{todonotes}
\usepackage[colorlinks=true, allcolors=blue]{hyperref}
\usepackage{adjustbox}
\newcommand{\R}{\mathbb R}
\newcommand{\restr}[1]{\lower0.4ex\hbox{$\vert$}\lower0.7ex\hbox{ $\!_{#1}$ }}

\newtheorem{theorem}{Theorem}[section]
\newtheorem{lemma}[theorem]{Lemma}
\newtheorem{corollary}[theorem]{Corollary}
\newtheorem{proposition}[theorem]{Proposition}

\newtheorem{remark}[theorem]{Remark}

\numberwithin{equation}{section}

\title{A note on surjectivity of piecewise affine mappings}
\author{Manuel Radons}

\begin{document}
	\maketitle
	
	\begin{abstract}
		A standard theorem in nonsmooth analysis states that a piecewise affine function 
		$F:\mathbb R^n\rightarrow\mathbb R^n$ is surjective if it is coherently oriented 
		in that the linear parts of its selection functions all have the same nonzero determinant sign.
		In this note we prove that surjectivity already follows from coherent orientation of the selection functions which are active on the unbounded sets of a polyhedral subdivision of the domain corresponding to $F$. A side bonus of the argumentation is a short proof of the classical statement that an injective piecewise affine function is coherently oriented.
	\end{abstract}
	
	\section{Introduction}
	\label{intro}
	Throughout, we assume familiarity with basic polyhedral terminology as described, e.g., in \cite{ziegler1993polytopes}.
	
	A continuous function $F:\mathbb R^n \rightarrow \mathbb R^m$ is called \textit{piecewise affine} if there exists a \textit{finite} set of affine functions $F_i(x)= A_ix+b_i$, such that $F$ coincides with an $F_i$ for every $x\in \mathbb R^n$ \cite[p. 15ff]{scholtes2012introduction}. The $F_i$ are called \textit{selection functions}. If $F$ coincides with $F_i$ on a set $U\subset\mathbb R^n$, we say that $F_i$ is \textit{active} on $U$.
	
	Any piecewise affine function $F:\mathbb R^n \rightarrow \mathbb R^m$ admits a corresponding (nonunique) partition $\mathcal P(F)$ of $\mathbb R^n$ into \textit{nonempty} (thus $n$-dimensional) convex polyhedra such that \cite[p. 28]{scholtes2012introduction}:
	\begin{enumerate}
		\item For every polyhedron $P_k\in\mathcal P(F)$ there exists a selection function $F_i$, such that $F\restr{P_k}=F_i$.
		\item The intersection of two polyhedra $P_k, P_l\in\mathcal P(F), k\ne l,$ is either empty, or a common \textit{proper} face of $P_k$ and $P_l$. 
		\item If $P_k\cap P_l\ne\emptyset$, then the selection functions which are active on $P_k$ and $P_l$, respectively, coincide on $P_k\cap P_l$.  
	\end{enumerate}
	We denote by $\mathcal P^\circ(F)$ the set of polytopes, i.e., of compact polyhedra, in $\mathcal P(F)$, and by $\mathcal P^\cup(F)$ the unbounded polyhedra in $\mathcal P(F)$.
	
	It is well known that a piecewise affine function $F:\mathbb R^n \rightarrow \mathbb R^n$ is surjective if it is \textit{coherently oriented} in that the linear parts of its selection functions all have the same nonzero determinant sign (see, e.g., \cite[Prop. 2.3.5, p. 34]{scholtes2012introduction} and \cite[Prop. 2.3.6, p. 35]{scholtes2012introduction}). 
	Clearly, for a piecewise affine function $F:\R\rightarrow\R$ with $\vert\mathcal P^\circ(F)\vert =: p>0$ we have $\vert\mathcal P^\cup(F)\vert =2$. Then via elementary arguments it can easily be verified that $F$ is surjective if and only if both affine functions which are active on the rays in $\mathcal P^\cup(F)$ have a positive or a negative slope, respectively. That is, if they are coherently oriented. (Compare figures below.)
	
	Topologically, this means that the surjectivity of $F$ (or the lack thereof) is independent of its behavior on the polytopes in $\mathcal P^\circ(F)$, so long as $F$ is continuous. Algorithmically, it means that to check for surjectivity of $F$ we need to consider the slopes of exactly two selection functions, instead of $p$ many, where $p$ may be arbitrarily large.

	\trimbox{0.1cm -0.5cm -0.5cm -0.5cm}
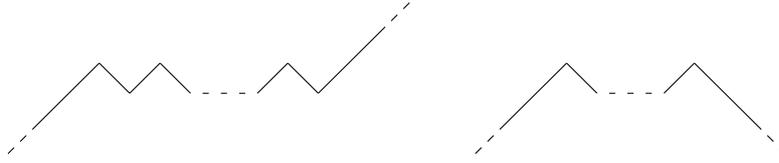
{ 
		\begin{tikzpicture}[scale=.8]
		
		\draw (0.4,0.4) -- (1.5,1.5);
		\draw (0,0) -- (0.1,0.1);
		\draw (0.2,0.2) -- (0.3,0.3);
		\draw (1.5,1.5) -- (2,1);
		\draw (2,1) -- (2.5,1.5);
		\draw (2.5,1.5) -- (3.0,1);
		\draw (3.2, 1) -- (3.3, 1);
		\draw (3.5, 1) -- (3.6, 1);
		\draw (3.8, 1) -- (3.9, 1);
		\draw (4.1, 1) -- (4.6, 1.5);
		\draw (4.6, 1.5) -- (5.1,1);
		\draw (6.6,2.5) -- (6.5,2.4);
		\draw (6.4,2.3) -- (6.3,2.2);
		\draw (5.1,1) -- (6.2,2.1);
		
		\end{tikzpicture} }
	\trimbox{0cm -0.5cm 0cm -0.5cm}{ 
		\begin{tikzpicture}[scale=.8]
		
		\draw (0.4,0.4) -- (1.5,1.5);
		\draw (0,0) -- (0.1,0.1);
		\draw (0.2,0.2) -- (0.3,0.3);
		\draw (1.5,1.5) -- (2,1);
		\draw (2.2, 1) -- (2.3, 1);
		\draw (2.5, 1) -- (2.6, 1);
		\draw (2.8, 1) -- (2.9, 1);
		\draw (3.1, 1) -- (3.6, 1.5);
		\draw (5.1,0) -- (5,0.1);
		\draw (4.9,0.2) -- (4.8,0.3);
		\draw (3.6, 1.5) -- (4.7,0.4);
		
		\end{tikzpicture}
	}

	In this note we prove an analogous result for arbitrary dimension $n$ (albeit, without the "only if"). 
	In Section \ref{section:degree-basics} we will assemble the necessary prerequisites from the literature. The main result is proved in Section \ref{section:main-result}. The techniques employed also yield a simple proof of the well-known statement that an injective piecewise affine function is coherently oriented (and thus surjective).

	\section{Mapping Degree Basics}\label{section:degree-basics}
	
	The following definitions and facts can be found, e.g., in \cite[p. 111ff]{ruiz2009deg}. Let $f:\mathbb R^n:\rightarrow\mathbb R^n$ be a continuous function, $\Omega\subset\mathbb R^n$ a bounded domain, and let $y\in\mathbb R^n\setminus f(\partial\bar\Omega)$, where $\bar\Omega$ is the closure of $\Omega$ and $\partial\bar\Omega$ denotes the boundary of $\bar\Omega$. We say $y$ is a \textit{regular value} of $f\restr{\Omega}$ if either $(f\restr{\Omega})^{-1}(y)=\emptyset$ or if the differential $D_xf$ of all $x\in f^{-1}(y)$ exists and is invertible. The \textit{local degree} of $y$ on $\Omega$ is denoted by $\operatorname{deg}(f,\Omega,y)$. We will need the following two of its properties:
	\begin{enumerate}
		\item $\operatorname{deg}(f, \Omega,y) = \sum_{x\in (f\restr{\Omega})^{-1}(y)}\operatorname{sign}[\det(D_xF)]$ -- which especially implies that \[(f\restr{\Omega})^{-1}(y)\ne\emptyset\quad\quad \text{ if }\quad\quad \operatorname{deg}(f, \Omega,y)\ne 0\,.\] 
		\item \textit{Nearness property:} Let $y, y'$ be regular values of $f\restr{\Omega}$. If \[\operatorname{dist}(y,y')\ <\ \operatorname{dist}(y,f(\partial\bar\Omega))\,,\] then $\operatorname{deg}(f, \Omega,y)=\operatorname{deg}(f, \Omega,y')$.
		\item The regular values of $f$ are dense in the codomain.
	\end{enumerate}
	Also note that for a piecewise affine function $F:\R^n\rightarrow\R^n$ the preimage $F^{-1}(y)$ is finite and discrete for any regular value $y$ of $F$.
	
	\section{Main result}\label{section:main-result}
	
	Let $F: \mathbb R^n\rightarrow \mathbb R^n$ be a piecewise affine function and $\mathcal P(F)$ a corresponding subdivision of $\R^n$. We say $F$ is nonsingular at infinity if all selection functions which are active on unbounded polyhedra of $\mathcal P(F)$ have a nonsingular linear part.
	
	\begin{lemma}(Globally defined degree)\label{lem:global-degree}
		Let $F: \mathbb R^n\rightarrow \mathbb R^n$ be a piecewise affine function which is nonsingular at infinity. Moreover, let $y,y'\in\R^n$ be two regular values of $F$, and $\Omega$ a bounded domain that contains their preimage. Then 
		\begin{align*}
		\operatorname{deg}(F,\Omega,y)\ =\ \operatorname{deg}(F,\Omega,y')\,,
		\end{align*}
		i.e., the mapping degree is globally defined for all regular values of $F$.
	\end{lemma}
	\begin{proof}
		Denote by $S^{n-1}_R$ the sphere of radius $R$ with respect to the Euclidean norm $\Vert\cdot\Vert$. 
		By hypothesis, the selection functions which are active on unbounded polyhedra of $\mathcal P(F)$ are affine homeomorphisms. Hence, $\Vert x\Vert\rightarrow\infty$ implies $\Vert F(x)\Vert\rightarrow\infty$.
		Conequently, for sufficiently large 
		\begin{align*}
		R\ >\ \max\left(\max_{x\in F^{-1}(y)}(\Vert x\Vert),\max_{x\in F^{-1}(y')}(\Vert x\Vert)\right)
		\end{align*}
		we can achieve 
		\begin{align*}
		\min_{x\in S^{n-1}_R}(\Vert F(x)\Vert)\ >\ 2(\Vert y\Vert+\Vert y'\Vert)\,.
		\end{align*}
		In this constellation we have 
		\begin{align*}
		\operatorname{dist}\left[y,F(\partial\bar B(0,R))\right]\ >\ \operatorname{dist}\left[y,y'\right]\,,
		\end{align*}
		where $\bar B(0,R)$ denotes the closed ball of radius $R$ centered at $0$, whose boundary is $S^{n-1}_R$. Hence, the nearness property of the mapping degree yields
		\begin{align*}
		\operatorname{deg}\left[F, B(0,R), y'\right]\ =\ \operatorname{deg}\left[F, B(0,R), y\right]\,.
		\end{align*}
	\end{proof}
	Lemma \ref{lem:global-degree} justifies to talk of \textit{the degree} of $F$, if it is nonsingular at infinity.
	\begin{corollary}\label{cor:surjective}
		Let $F: \mathbb R^n\rightarrow \mathbb R^n$ be a piecewise affine function which is nonsingular at infinity and has nonzero degree. Then $F$ is surjective.
	\end{corollary}
	\begin{proof}
		Lemma \ref{lem:global-degree} asserts that all regular values of $F$, which are dense in $\R^n$, are contained in its image. 
		But, since piecewise affine functions are closed \cite[p. 42]{scholtes2012introduction}, the image of $F$ is its own closure.
	\end{proof}
	\begin{corollary}\label{cor:branched-covering}
		Let $F:\R^n\rightarrow\R^n$ be a coherently oriented piecewise affine function. Then all regular values of $F$ have the same number of preimages. 
	\end{corollary}
	\begin{proof}
		$F$ is nonsingular at infinity, thus its degree is globally defined. Then the statement follows from the fact that all selection functions have the same nonzero determinant sign. 
	\end{proof}
	\begin{remark}
		Corollary \ref{cor:branched-covering} especially implies that a coherently oriented piecewise affine function is a branched covering.
	\end{remark}
	\begin{lemma}\label{lem:reg-val}
		Let $F: \mathbb R^n\rightarrow \mathbb R^n$ be a piecewise affine function which is nonsingular at infinity, and $\mathcal P(F)$ some corresponding subdivision of $\R^n$. Then the image of $F$ contains a regular value whose preimage lies exclusively in the unbounded polyhedra of $\mathcal P(F)$.
	\end{lemma}
	\begin{proof}
		Since $F$ is closed, the image of the compact set $\mathcal P^\circ(F)$ is compact, and 
		\begin{align*}
		r\ :=\ \max_{y\in F(\mathcal P^\circ(F))}(\Vert y\Vert)\, 
		\end{align*}
		is well defined.
		As $F$ is nonsingular at infinity, there exists some closed ball $B(x,\varepsilon)$ in the interior of some unbounded polyhedron of $\mathcal P(F)$ such that $\min_{x\in B(x,\varepsilon)}(\Vert F(x)\Vert)>r$. Then the interior of $F(B(x,\varepsilon))$ contains a regular value, whose preimage lies exclusively in unbounded polyhedra of $\mathcal{P}(F)$.
	\end{proof}
	
	\begin{proposition}\label{Main}
		Let $F: \mathbb R^n\rightarrow \mathbb R^n$ be a piecewise affine function and $\mathcal P(F)$ a corresponding subdivision of $\R^n$. Then $F$ is surjective if the linear parts of all selection functions which are active on unbounded polyhedra in $\mathcal P(F)$ have the same nonzero determinant sign. 
	\end{proposition}
	\begin{proof}
		Lemma \ref{lem:reg-val} asserts that the image of $F$ contains a regular value $y$ whose preimage lies exclusively in unbounded polyhedra of $\mathcal{P}(F)$. Then \[\sum_{x\in F^{-1}(y)}\operatorname{sign}[\operatorname{det}(D_xF)]\] 
		cannot be zero since the determinant signs of the differentials $D_xF$ are either all positive or all negative, which means $y$ has nonzero degree. The claim now follows from Lemma \ref{lem:global-degree} and Corollary \ref{cor:surjective}.
	\end{proof}    
	\begin{remark}
		Note that the statement of Proposition \ref{Main} still holds if we merely assume continuity on $\mathcal P^\circ(F)$, but not necessarily affinity. 
	\end{remark}
	As a side-bonus of the argumentation so far, we obtain a short proof for another classical statement about piecewise affine functions:
	\begin{proposition}\label{prop:inj-coh}
		An injective piecewise affine function $F:\R^n\rightarrow\R^n$ is coherently oriented.
	\end{proposition}
	\begin{proof}
		Sice each selection function is active on some full dimensional polyhedron, the injectivity of $F$ implies that all its selection functions have an invertible linear part, which especially means that $F$ is nonsingular at infinity. Thus its degree is globally defined (Lemma \ref{lem:global-degree}).  Moreover, we can find some regular value in the image (e.g. due to Lemma \ref{lem:reg-val}), which then has degree either $+1$ or $-1$, since it has only one preimage. Hence, the (global) degree of $F$ is either $+1$ or $-1$, respectively, which yields the surjectivity of $F$. The injectivity of $F$ asserts that the linear parts of all selection function have the same determinant sign.
	\end{proof}
	\begin{corollary}
		Let $F:\R^n\rightarrow\R^n$ be a piecewise affine function. Then the following are equivalent.
		\begin{enumerate}
			\item $F$ is coherently oriented of degree $1$.
			\item $F$ is bijective.
			\item $F$ is a homeomorphism.
		\end{enumerate}
	\end{corollary}
	\begin{proof}
		The equivalence of $1.$ and $2.$ is a direct consequnce of Proposition \ref{prop:inj-coh}. The implication $3.\Rightarrow 2.$ is clear. To prove $2.\Rightarrow 3.$, it suffices to show that a bijective (and thus coherently oriented) piecewise affine function is open. For a piecewise affine function $F$, coherent orientation implies the coherent orientation of the directional derivative of $F$ at all points in the domain. But this is equivalent to the metric regularity of $F$ \cite[Thm.214]{fusek2013open}, which is well known to imply openness.   
	\end{proof}

	
	

\end{document}